\documentclass[12pt]{amsart}
\usepackage{amsmath,amsthm,amssymb,fullpage}

\usepackage[matrix,arrow,curve]{xy}
\makeatletter\@addtoreset{equation}{section} \makeatother

\newtheorem{theorem}[equation]{Theorem}
\newtheorem{lemma}[equation]{Lemma}
\newtheorem{proposition}[equation]{Proposition}

\theoremstyle{definition}
\newtheorem{example}[equation]{Example}
\newtheorem{definition}[equation]{Definition}

\theoremstyle{remark}
\newtheorem{remark}[equation]{Remark}

\def \P {\mathbb{P}}
\def \C {\mathbb{C}}
\def \Q {\mathbb{Q}}
\def \lct {\mathrm{lct}}
\def \PSL {\mathrm{PSL}}
\def \Aut {\mathrm{Aut}}

\def \ge {\geqslant}
\def \le {\leqslant}

\title{EXTREMAL METRICS ON DEL PEZZO THREEFOLDS}%
\author{Ivan Cheltsov and Constantin Shramov}

\dedicatory{In memory of Vasily Iskovskikh}
\thanks{The first author was supported by
the grants NSF DMS-0701465 and EPSRC EP/E048412/1,
the~second~author was supported by the grants RFFI
No.~08-01-00395-a, N.Sh.-1987.2008.1 and EPSRC EP/E048412/1.}

\begin{document}

\begin{abstract}
We prove the existence of K\"ahler--Einstein metrics on a
nonsingular section of the Grassmannian $\mathrm{Gr}(2,
5)\subset\mathbb{P}^9$ by a linear subspace of codimension $3$,
and the Fermat hypersurface of degree $6$ in
$\mathbb{P}(1,1,1,2,3)$. We also show that a global log canonical
threshold of the Mukai--Umemura variety is equal to $1/2$.
\end{abstract}

\maketitle

\section{Introduction}

Let $X$ be a variety\footnote{All varieties are assumed to be
complex, algebraic, projective and normal.} with at most log
canonical singularities (see  \cite{Ko97}), and let $D$ be an
effective $\mathbb{Q}$-Cartier $\mathbb{Q}$-divisor on the variety
$X$. Then the number
$$
\mathrm{lct}\big(X,D\big)=\mathrm{sup}\left\{\lambda\in\mathbb{Q}\
\Big|\ \text{log pair}\ \big(X, \lambda D\big)\
\text{log canonical}\right\}\in\mathbb{Q}\cup\big\{+\infty\big\}%
$$
is called the log canonical threshold of the divisor $D$ (see
\cite{Sho93}).

Suppose that $X$ is a Fano variety with at most log terminal
singularities (see \cite{IsPr99}).

\begin{definition}
\label{definition:threshold} Global log canonical threshold of the
Fano variety $X$ is the number
$$
\mathrm{lct}\big(X\big)=\mathrm{inf}\left\{\mathrm{lct}\big(X,D\big)\ \Big\vert\ D\ \text{is an effective $\mathbb{Q}$-divisor on $X$ such that}\ D\ \sim_{\mathbb{Q}} -K_{X}\right\}\geqslant 0.%
$$
\end{definition}

Recall that every Fano variety $X$ is rationally connected (see
\cite{Zh06}). Thus, the group $\mathrm{Pic}(X)$ is torsion free.
Hence
$$
\mathrm{lct}\big(X\big)=\mathrm{sup}\left\{\lambda\in\mathbb{Q}\ \left|\ %
\aligned
&\text{the log pair}\ \Big(X, \lambda D\Big)\ \text{is log canonical}\\
&\text{for every effective $\mathbb{Q}$-divisor}\ D\sim_{\mathbb{Q}} -K_{X}\\
\endaligned\right.\right\}.
$$

\begin{example}
\label{example:Cheltsov-Park} Let $X$ be a smooth hypersurface in
$\mathbb{P}^{n}$ of degree $m$, where $2\le m\le n$. Then
$$
\mathrm{lct}\big(X\big)=\frac{1}{n+1-m}
$$
if $m<n$ (see  \cite{Ch01b}). Thus, we have
$\mathrm{lct}(\mathbb{P}^{n})=1/(n+1)$. Suppose that $n=m$. by
\cite{Ch01b}
$$
1\geqslant\mathrm{lct}\big(X\big)\geqslant \frac{n-1}{n}.%
$$
It follows from \cite{Pu04d} and \cite{ChPaWo} that if $X$ is
general, then
$$
\mathrm{lct}\big(X\big)\geqslant\left\{%
\aligned
&1,\ \text{if}\ n\geqslant 6,\\%
&22/25,\ \text{if}\ n=5,\\%
&16/21,\ \text{if}\ n=4,\\%
&3/4,\ \text{if}\ n=3.\\%
\endaligned\right.%
$$
One has $\mathrm{lct}(X)=1-1/n$ if $X$ contains a cone of
dimension $n-2$.
\end{example}

\begin{example}
\label{example:Hwang} Let $X$ be a rational homogeneous space such
that $-K_{X}\sim rD$ and
$$
\mathrm{Pic}\big(X\big)=\mathbb{Z}\big[D\big],
$$
where $D$ is an ample divisor and $r\in\mathbb{Z}_{>0}$. Then
$\mathrm{lct}(X)=1/r$ (see \cite{Hw06b}).
\end{example}

\begin{example}
\label{example:IHES} Let $X$ be a quasismooth hypersurface in
$\mathbb{P}(1,a_{1},a_{2},a_{3},a_{4})$ of
degree~$\sum_{i=1}^{4}a_{i}$ such that $X$ has at most terminal
singularities, where $a_{1}\leqslant a_{2}\leqslant a_{3}\leqslant
a_{4}$. Then
$$
-K_{X}\sim_{\mathbb{Q}}\mathcal{O}_{\mathbb{P}(1,\,a_{1},\,a_{2},\,a_{3},\,a_{4})}\big(1\big)\Big\vert_{X}
$$
and there are $95$ possibilities for the quadruple
$(a_{1},a_{2},a_{3},a_{4})$ (see \cite{IF00}). One has
$$
1\geqslant\mathrm{lct}\big(X\big)\geqslant\left\{%
\aligned
&16/21\ \text{if}\ a_{1}=a_{2}=a_{3}=a_{4}=1,\\%
&7/9\ \text{if}\ (a_{1},a_{2},a_{3},a_{4})=(1,1,1,2),\\%
&4/5\ \text{if}\ (a_{1},a_{2},a_{3},a_{4})=(1,1,2,2),\\%
&6/7\ \text{if}\ (a_{1},a_{2},a_{3},a_{4})=(1,1,2,3),\\%
&1\ \text{in the remaining cases},\\%
\endaligned\right.%
$$
if $X$ is general (see \cite{Ch07a}, \cite{ChPaWo}, \cite{Ch08d}).
\end{example}

\begin{example}
\label{example:del-Pezzos} Let $X$ be smooth  del Pezzo surface.
It follows from \cite{Ch07b} that
$$
\mathrm{lct}\big(X\big)=\left\{%
\aligned
&1\ \text{if}\ K_{X}^{2}=1\ \text{and}\ |-K_{X}|\ \text{contains no cuspidal curves},\\%
&5/6\ \text{if}\ K_{X}^{2}=1\ \text{and}\ |-K_{X}|\ \text{contains a cuspidal curve},\\%
&5/6\ \text{if}\ K_{X}^{2}=2\ \text{and}\ |-K_{X}|\ \text{contains no tacnodal curves},\\%
&3/4\ \text{if}\ K_{X}^{2}=2\ \text{and}\ |-K_{X}|\ \text{contains a tacnodal curve},\\%
&3/4\ \text{if}\ X\ \text{is a cubic in}\ \mathbb{P}^{3}\ \text{with no Eckardt points},\\%
&2/3\ \text{if either}\ X\ \text{is a cubic in}\ \mathbb{P}^{3}\ \text{with an Eckardt point, or  $K_{X}^{2}=4$},\\%
&1/2\ \text{if}\ X\cong\mathbb{P}^{1}\times\mathbb{P}^{1}\ \text{or}\ K_{X}^{2}\in\{5,6\},\\%
&1/3\ \text{in the remaining cases}.\\%
\endaligned\right.%
$$
\end{example}

Let $G\subset\mathrm{Aut}(X)$ be an arbitrary subgroup.

\begin{definition}
\label{definition:G-threshold} Global $G$-invariant log canonical
threshold $\lct(X, G)$ of the Fano variety $X$ is the number
$$
\mathrm{sup}\left\{\epsilon\in\mathbb{Q}\ \left|\ %
\aligned
&\text{the log pair}\ \left(X, \frac{\epsilon}{n}\mathcal{D}\right)\ \text{has log canonical singularities for every}\\
&\text{$G$-invariant linear system}\ \mathcal{D}\subset\big|-nK_{X}\big|\ \text{and every}\ n\in\mathbb{Z}_{>0}\\
\endaligned\right.\right\}.%
$$
\end{definition}

If the Fano variety $X$ is smooth and $G$ is compact, then it
follows from \cite[Appendix~A]{ChSh08} that
$$
\mathrm{lct}\big(X,G\big)=\alpha_{G}\big(X\big),
$$
where $\alpha_{G}(X)$ is the invariant introduced in \cite{Ti87}.
We have
$\mathrm{lct}(X)\leqslant\mathrm{lct}(X,G)\in\mathbb{R}\cup\{+\infty\}$.

\begin{remark}
Suppose that the subgroup $G$ is finite. Then
$$
\mathrm{lct}\big(X,
G\big)=\mathrm{sup}\left\{\lambda\in\mathbb{Q}\ \left|\ \aligned%
&\text{log pair}\ \left(X, \lambda D\right)\ \text{is log canonical for every}\\
&\text{effective $G$-invariant $\mathbb{Q}$-divisor}\ D\sim_{\mathbb{Q}} -K_{X}\\
\endaligned\right.\right\}.%
$$
Indeed, it is enough to show that if $\mathcal{D}\subset |-mK_X|$
is a $G$-invariant linear system such that the log pair $(X,
c\mathcal{D})$ is not log canonical for some $c\in\Q_{\geqslant
0}$, then there is a $G$-invariant effective $\Q$-divisor
$B\sim_{\mathbb{Q}} -mK_X$ such that the log pair $(X, cB)$ is not
log canonical. Put $k=|G|$. Suppose that the log pair $(X,
c\mathcal{D})$ is not log canonical. Let $D\in\mathcal{D}$ be a
general divisor. Then the log pair
$$
\big(X, \frac{c}{k}\sum_{g\in G} g(D)\big)
$$
is not log canonical as well (see the proof of
\cite[Theorem~4.8]{Ko97}), which implies the required assertion.
\end{remark}

\begin{example}
\label{example:Klein} The simple group
$\mathrm{PGL}(2,\mathrm{F}_{7})$ is a group of automorphisms of
the quartic
$$
x^{3}y+y^{3}z+z^{3}x=0\subset\mathbb{P}^{2}\cong\mathrm{Proj}\Big(\mathbb{C}[x,y,z]\Big),
$$
which induces an embedding
$\mathrm{PGL}(2,\mathrm{F}_{7})\subset\mathrm{Aut}(\mathbb{P}^{2})$.
One has $\mathrm{lct}(\mathbb{P}^{2},
\mathrm{PGL}(2,\mathrm{F}_{7}))=4/3$
(see~\cite{PrMar99},~\cite{Ch07b}).
\end{example}

\begin{example}
\label{example:Fermat-cubic} Let $X$ be the cubic surface in
$\mathbb{P}^{3}$ given by the equation
$$
x^{3}+y^{3}+z^{3}+t^{3}=0\subset\mathbb{P}^{3}\cong\mathrm{Proj}\Big(\mathbb{C}\big[x,y,z,t\big]\Big),
$$
and let $G=\mathrm{Aut}(X)\cong
\mathbb{Z}_{3}^{3}\rtimes\mathrm{S}_{4}$. Then $\mathrm{lct}(X,
G)=4$ by~\cite{Ch07b}.
\end{example}

The following result is proved in \cite{Ti87}, \cite{Na90},
\cite{DeKo01} (cf. \cite[Appendix~A]{ChSh08}).

\begin{theorem}
\label{theorem:KE} Suppose that $X$ has at most quotient
singularities, the group $G$ is compact, and the inequality
$$
\mathrm{lct}\big(X,G\big)>\frac{\mathrm{dim}\big(X\big)}{\mathrm{dim}\big(X\big)+1}%
$$
holds. Then $X$ admits an orbifold K\"ahler--Einstein metric.
\end{theorem}

\begin{remark}\label{remark:compact-vs-complex}
Let $G\subset\Aut(X)$ be a reductive subgroup, and $G'\subset G$ the
maximal compact subgroup of $G$. Then a restriction to $G'$ of any
irreducible representation of $G$ remains irreducible as a complex
representation of $G'$. This implies that all linear systems on $X$
that are invariant with respect to $G$ are also invariant with
respect to $G'$ (the converse holds by obvious reasons). In
particular, $\lct(X, G)=\lct(X, G')$.
\end{remark}

Theorem~\ref{theorem:KE} has many applications (see
Examples~\ref{example:Cheltsov-Park}, \ref{example:IHES},
\ref{example:Fermat-cubic}).

\begin{example}
\label{example:Fermat-hypersurface} Let $X$ be one of the
following smooth Fano varieties:
\begin{itemize}
\item a Fermat hypersurface in $\mathbb{P}^{n}$ of degree $n/2\leqslant d\leqslant n$ (cf. Example~\ref{example:Fermat-cubic});%
\item a smooth complete intersection of two quadrics in $\mathbb{P}^{5}$ that is given by%
$$
\sum_{i=0}^{5}x_{i}^{2}=\sum_{i=0}^{5}\zeta^{i}x_{i}^{2}=0\subseteq\mathbb{P}^{5}\cong\mathrm{Proj}\Big(\mathbb{C}\big[x_{0},\ldots,x_{5}\big]\Big),
$$
where $\zeta$ is a primitive sixth root of unity;%
\item a hypersurface in $\mathbb{P}(1^{n+1},q)$ of degree $pq$ that is given by the equation%
$$
w^{p}=\sum_{i=0}^{5}x_{i}^{pq}\subseteq\mathbb{P}\big(1^{n+1},q\big)\cong\mathrm{Proj}\Big(\mathbb{C}\big[x_{0},\ldots,x_{n},w\big]\Big),
$$
such that $pq-q\leqslant n$, where
$\mathrm{wt}(x_{0})=\ldots=\mathrm{wt}(x_{n})=1$,
$\mathrm{wt}(w)=q\in\mathbb{Z}_{>0}$ and $p\in \mathbb{Z}_{>0}$.
\end{itemize}
Let $G=\mathrm{Aut}(X)$. Then $G$ is finite, and the inequality
$\mathrm{lct}(X,G)\geqslant 1$ holds (see \cite{Ti87},
\cite{Na90}).
\end{example}

The numbers $\mathrm{lct}(X)$ and  $\mathrm{lct}(X, G)$ also play
an important role in birational geometry. For instance, the
following result holds (see  \cite{Ch07b}).

\begin{theorem}
\label{theorem:G-Pukhlikov} Let $X_{i}$ be a Fano variety, and let
$G_{i}\subset\mathrm{Aut}(X_{i})$ be a finite subgroup such that
\begin{itemize}
\item the variety $X_{i}$ is $G_{i}$-birationally superrigid (see  \cite{ChSh08}),%
\item the inequality $\mathrm{lct}(X_{i},G_{i})\geqslant 1$ holds,%
\end{itemize}
where $i=1,\ldots,r$. Then the following assertions hold:
\begin{itemize}
\item there is no $G_{1}\times\ldots\times G_{r}$-equivariant birational map $\rho\colon X_{1}\times\ldots\times X_{r}\dasharrow\mathbb{P}^{n}$;%
\item every $G_{1}\times\ldots\times G_{r}$-equivariant birational automorphism of $X_{1}\times\ldots\times X_{r}$ is biregular;%
\item for every $G_{1}\times\ldots\times G_{r}$-equivariant
rational dominant map
$$
\rho\colon X_{1}\times\ldots\times X_{r}\dasharrow Y,
$$
whose general fiber is a rationally connected variety, there a
commutative diagram
$$
\xymatrix{
X_{1}\times\ldots\times X_{r}\ar@{->}[d]_{\pi}\ar@{-->}[rrrrd]^{\rho}\\
X_{i_{1}}\times\ldots\times X_{i_{k}}\ar@{-->}[rrrr]_{\xi}&&&&Y}%
$$
where $\xi$ is a birational~map, $\pi$ is a natural projection,
and $\{i_{1},\ldots,i_{k}\}\subseteq\{1,\ldots,r\}$.
\end{itemize}
\end{theorem}

Varieties satisfying all hypotheses of
Theorem~\ref{theorem:G-Pukhlikov} do exist.

\begin{example}
\label{example:Valentiner} The simple group $\mathfrak{A}_{6}$ is
a group of automorphisms of the sextic
$$
10x^{3}y^{3}+9zx^{5}+9zy^{5}+27z^{6}=45x^{2}y^{2}z^{2}+135xyz^{4}\subset\mathbb{P}^{2}\cong\mathrm{Proj}\Big(\mathbb{C}\big[x,y,z\big]\Big),
$$
which induces an embedding
$\mathfrak{A}_{6}\subset\mathrm{Aut}(\mathbb{P}^{2})$. Then
$\mathbb{P}^{2}$ is $\mathfrak{A}_{6}$-birationally~super\-ri\-gid
and the equality $\mathrm{lct}(\mathbb{P}^{2},\mathfrak{A}_{6})=2$
holds (see  \cite{PrMar99}, \cite{Ch07b}). Thus, there is an
induced embedding
$$
\mathfrak{A}_{6}\times
\mathfrak{A}_{6}\cong\Omega\subset\mathrm{Bir}\big(\mathbb{P}^{4}\big)
$$
such that $\Omega$ is not conjugate to any subgroup in
$\mathrm{Aut}(\mathbb{P}^{4})$ by
Theorem~\ref{theorem:G-Pukhlikov}.
\end{example}

\bigskip

Let $V$ be a smooth Fano threefold (see \cite{IsPr99}) such that
$-K_{V}\sim 2H$, where $H$ is an ample Cartier divisor that is not
divisible in $\mathrm{Pic}(V)$.

\begin{remark}
\label{remark:del-Pezzo-threefold} The variety $V$ is called a del
Pezzo variety, since a general element in the linear system $|H|$ is
a smooth del Pezzo surface.
\end{remark}

It is well-known that $V$ is one of the following varieties:
\begin{itemize}
\item $V_1$, i.\,e. a hypersurface in $\mathbb{P}(1,1,1,2,3)$ of degree $6$;%
\item $V_2$, i.\,e. a hypersurface in $\mathbb{P}(1,1,1,1,2)$ of degree $4$;%
\item $V_3$, i.\,e. a cubic surface in $\mathbb{P}^3$;%
\item $V_4$, i.\,e. a complete intersection of two quadrics in $\mathbb{P}^5$;%
\item $V_5$, i.\,e. a section of the Grassmannian
$\mathrm{Gr}(2, 5)\subset\mathbb{P}^9$
by a linear subspace of codimension $3$ (all such sections are isomorphic);%
\item $W$, a divisor in $\mathbb{P}^{2}\times\mathbb{P}^{2}$ of bidegree $(1,1)$;%
\item $V_7$, i.\,e. a blow-up of $\mathbb{P}^3$ at a point;%
\item the product $\mathbb{P}^1\times\mathbb{P}^1\times\mathbb{P}^1$.%
\end{itemize}

\begin{remark}
In \cite{ChSh08} the values of global log canonical thresholds of
smooth del Pezzo threefolds were found:
$$
\mathrm{lct}\big(V\big)=\left\{%
\aligned
&1/4,\ \text{if}\ V\ \text{is a blow-up of $\mathbb{P}^3$ at a point},\\%
&1/2\ \text{in the remaining cases}.\\%
\endaligned\right.%
$$
\end{remark}

\label{remark:KE-del-Pezzo-treefolds} Concerning K\"ahler--Einstein
metrics on $V$, the following is known:
\begin{itemize}
\item $V_7$ does not admit a K\"ahler--Einstein metric (see~\cite{WaZhu04});%
\item $V_4$ admits a  K\"ahler--Einstein metric (see~\cite{ArGhPi06}, cf. Example~\ref{example:Fermat-hypersurface});%
\item $W$ and $\mathbb{P}^1\times\mathbb{P}^1\times\mathbb{P}^1$
admit K\"ahler--Einstein metrics, since their automorphism groups
are reductive and act
on them transitively (see Theorem~\ref{theorem:KE} and Remark~\ref{remark:compact-vs-complex});%
\item there are examples of varieties
$V_2\subset\mathbb{P}(1,1,1,1,2)$ and $V_3\subset\mathbb{P}^{4}$
with large automorphism groups (see
Example~\ref{example:Fermat-hypersurface}) that admit
K\"ahler--Einstein metrics.%
\end{itemize}
The question of existence of K\"ahler--Einstein metrics on the
varieties $V_{1}$ and $V_{5}$ has not been studied in the literature
yet (cf. a remark before \cite[Theorem~3.2]{ArGhPi06}).

The main purpose of this paper is to prove the following
assertions.

\begin{theorem}\label{theorem:V5-lct}
Let $G$ be a maximal compact subgroup in $\mathrm{Aut}(V_5)$. Then
$$
\mathrm{lct}\big(V_5,\ G\big)=\lct(V_5, \Aut(V_5))=5/6.
$$
\end{theorem}

\begin{theorem}
\label{theorem:V1} Let $V_{1}$ be a hypersurface in
$\mathbb{P}(1,1,1,2,3)$, given by an equation
$$
w^2=t^3+x^6+y^6+z^6\subset\mathbb{P}\big(1,1,1,2,3\big)\cong\mathrm{Proj}\Big(\mathbb{C}\big[x,y,z,t,w\big]\Big),
$$
where $\mathrm{wt}(x)=\mathrm{wt}(y)=\mathrm{wt}(z)=1$,
$\mathrm{wt}(t)=2$ and $\mathrm{wt}(w)=3$. Then
$\mathrm{lct}(V_{1}, \mathrm{Aut}(V_{1}))\geqslant 1$.
\end{theorem}

Note that the latter results combined with Theorem~\ref{theorem:KE}
imply the existence of K\"ahler--Einstein metrics on the variety
$V_{5}$ and on the Fermat hypersurface of degree $6$ in
$\mathbb{P}(1,1,1,2,3)$.

\begin{remark}
\label{remark:V1-rigid} Let $V_{1}$ be a smooth hypersurface in
$\mathbb{P}(1,1,1,2,3)$ of degree $6$. Assume that
$\mathrm{lct}(V_{1},G)\geqslant 1$, where $G$ is a subgroup in
$\mathrm{Aut}(V_{1})$. Then
\begin{itemize}
\item the linear system $|H|$ does not contain $G$-invariant surfaces,%
\item the linear system $|H|$ does not contain $G$-invariant pencils (cf. the proof
of~\cite[Theorem~1.2]{PrMar99}),%
\item the variety $V_{1}$ is $G$-birationally superrigid (see \cite{Gr03}, \cite{Gr04}).%
\end{itemize}
\end{remark}

\begin{remark}
The methods we use to prove Theorem~\ref{theorem:V1-lct} are
similar to those of~\cite{Na90}. Nevertheless, some statements
of~\cite{Na90} (say,~\cite[Corollary~4.2]{Na90}, or a standard
method to exclude zero-dimensional components of a subscheme of
log canonical singularities) cannot be directly   applied in our
case, since the group $\Aut(V_1)$ \emph{never} acts on $V_1$
without fixed points.
\end{remark}

The structure of the paper is as follows.
Section~\ref{section:preliminaries} contains some auxiliary
statements. In Section~\ref{section:V1} we prove
Theorem~\ref{theorem:V1}. In Section~\ref{section:V5} we prove
Theorem~\ref{theorem:V5-lct}. The methods of
Section~\ref{section:V5} can be applied without significant
changes to one more interesting Fano threefold, the so-called
Mukai--Umemura variety (see \cite{Do07} and
Remark~\ref{remark:Donaldson}). To complete the picture in
Section~\ref{section:MU} we calculate the global log canonical
threshold of Mukai--Umemura variety without any group action.

We are grateful to A.\,Kuznetsov and E.\,Smirnov for useful
discussions.

\section{Preliminaries}
\label{section:preliminaries}

Let $X$ be a variety with log terminal singularities. Let us
consider a $\mathbb{Q}$-divisor
$$
B_{X}=\sum_{i=1}^{r}a_{i}B_{i},
$$
where $B_{i}$ is a prime Weil divisor on the variety $X$, and
$a_{i}$ is an arbitrary non-negative rational number. Suppose that
$B_{X}$ is a $\mathbb{Q}$-Cartier divisor such that $B_{i}\neq
B_{j}$ for $i\neq j$.

Let $\pi\colon\bar{X}\to X$ be a birational morphism such that
$\bar{X}$ is smooth. Put
$$
B_{\bar{X}}=\sum_{i=1}^{r}a_{i}\bar{B}_{i},
$$
where $\bar{B}_{i}$ is a proper transform of the divisor $B_{i}$
on the variety $\bar{X}$. Then
$$
K_{\bar{X}}+B_{\bar{X}}\sim_{\mathbb{Q}}\pi^{*}\Big(K_{X}+B_{X}\Big)+\sum_{i=1}^{n}c_{i}E_{i},
$$
where $c_{i}\in\mathbb{Q}$, and $E_{i}$ is an exceptional divisor
of the morphism $\pi$. Suppose that
$$
\left(\bigcup_{i=1}^{r}\bar{B}_{i}\right)\bigcup\left(\bigcup_{i=1}^{n}E_{i}\right)
$$
is a divisor with simple normal crossings. Put
$$
B^{\bar{X}}=B_{\bar{X}}-\sum_{i=1}^{n}c_{i}E_{i}.
$$

\begin{definition}
\label{definition:log canonical-singularities} The singularities
of $(X, B_{X})$ are log canonical (resp., log terminal) if
\begin{itemize}
\item the inequality $a_{i}\leqslant 1$ holds (resp., the
inequality $a_{i}<1$
holds),%
\item the inequality $c_{j}\geqslant -1$ holds (resp., the
inequality
$c_{j}>-1$ holds),%
\end{itemize}
for every $i=1,\ldots,r$ and $j=1,\ldots,n$.
\end{definition}

One can show that Definition~\ref{definition:log
canonical-singularities} does not depend on the choice of the
morphism $\pi$. Put
$$
\mathrm{LCS}\Big(X, B_{X}\Big)=\left(\bigcup_{a_{i}\geqslant
 1}B_{i}\right)\bigcup\left(\bigcup_{c_{i}\leqslant -1}\pi\big(E_{i}\big)\right)\subsetneq X,%
$$
and let us call $\mathrm{LCS}(X,B_{X})$ the locus of log canonical
singularities of the log pair $(X, B_{X})$.

\begin{definition}
\label{definition:log canonical-center} A proper irreducible
subvariety $Y\subsetneq X$ is said to be a center of log canonical
singularities of the log pair $(X, B_{X})$ if one of the following
conditions is satisfied:
\begin{itemize}
\item either the inequality $a_{i}\geqslant 1$ holds and $Y=B_{i}$,%
\item or the inequality $c_{i}\leqslant -1$ holds and $Y=\pi(E_{i})$,%
\end{itemize}
for some choice of the birational morphism $\pi\colon\bar{X}\to
X$.
\end{definition}

Let $\mathbb{LCS}(X, B_{X})$ be the set of all centers of log
canonical singularities of $(X, B_{X})$. Then
$$
Y\in\mathbb{LCS}\Big(X, B_{X}\Big)\Longrightarrow Y\subseteq\mathrm{LCS}\Big(X, B_{X}\Big)%
$$
and $\mathbb{LCS}(X, B_{X})=\varnothing\iff\mathrm{LCS}(X,
B_{X})=\varnothing$ $\iff$ the log pair $(X,B_{X})$ is log
terminal.

\begin{remark}
\label{remark:hyperplane-reduction} We can use similar
constructions and notation for any log pair $(X,
\lambda\mathcal{D})$, where $\mathcal{D}$ is a linear system, and
$\lambda$ is a non-negative rational number.
\end{remark}

The set $\mathrm{LCS}(X,B_{X})$ can be naturally equipped with a
scheme structure (see \cite{Na90}, \cite{Sho93}). Put
$$
\mathcal{I}\Big(X, B_{X}\Big)=\pi_{*}\Big(\sum_{i=1}^{n}\lceil
c_{i}\rceil
 E_{i}-\sum_{i=1}^{r}\lfloor a_{i}\rfloor \bar{B}_{i}\Big),%
$$
and let $\mathcal{L}(X, B_{X})$ be a subscheme that corresponds to
the ideal sheaf $\mathcal{I}(X, B_{X})$.

\begin{remark}
The scheme $\mathcal{L}(X, B_{X})$ is usually called the subscheme
of log canonical singularities of the log pair $(X,B_{X})$, and
the ideal sheaf $\mathcal{I}(X, B_{X})$ is usually called the
multiplier ideal sheaf of the log pair $(X,B_{X})$.
\end{remark}

It follows from the construction of the subscheme $\mathcal{L}(X,
B_{X})$ that
$$
\mathrm{Supp}\Big(\mathcal{L}\big(X,
B_{X}\big)\Big)=\mathrm{LCS}\Big(X, B_{X}\Big)\subset X.
$$

The following result is known as the Shokurov vanishing theorem
(see \cite{Sho93}) or the Nadel vanishing theorem
(see~\cite[Theorem~9.4.8]{La04}).

\begin{theorem}
\label{theorem:Shokurov-vanishing} Let $H$ be a nef and big
$\mathbb{Q}$-divisor on $X$ such that
$$
K_{X}+B_{X}+H\sim_{\mathbb{Q}} D
$$
for some Cartier divisor $D$ on the variety $X$. Then for every
$i\geqslant 1$
$$
H^{i}\Big(X,\ \mathcal{I}\big(X, B_{X}\big)\otimes D\Big)=0.
$$
\end{theorem}

The following result is known as the Shokurov connectedness
theorem.

\begin{theorem}
\label{theorem:connectedness} Suppose that $-(K_{X}+B_{X})$ is nef
and big. Then $\mathrm{LCS}(X, B_{X})$ is connected.
\end{theorem}

\begin{proof}
Put $D=0$. Then it follows from
Theorem~\ref{theorem:Shokurov-vanishing} that the sequence
$$
\mathbb{C}=H^{0}\Big(\mathcal{O}_{X}\Big)\longrightarrow
H^{0}\Big(\mathcal{O}_{\mathcal{L}(X,\,
 B_{X})}\Big)\longrightarrow H^{1}\Big(\mathcal{I}\big(X, B_{X}\big)\Big)=0%
$$
is exact. Thus, the locus
$$
\mathrm{LCS}\Big(X,\ B_{X}\Big)=\mathrm{Supp}\Big(\mathcal{L}\big(X,\ B_{X}\big)\Big)%
$$
is connected.
\end{proof}

One can generalize Theorem~\ref{theorem:connectedness} in the
following way (see \cite[Lemma~5.7]{Sho93}).

\begin{theorem}
\label{theorem:log adjunction-connectedness-theorem} Let
$\psi\colon X\to Z$ be a morphism. Then the set
$$
\mathrm{LCS}\Big(\bar{X},\ B^{\bar{X}}\Big)
$$
is connected in a neighborhood of every fiber of the morphism
$\psi\circ\pi\colon X\to Z$ in the case when
\begin{itemize}
\item the morphism $\psi$ is surjective and has connected fibers,%
\item the divisor $-(K_{X}+B_{X})$ is nef and big with respect to $\psi$.%
\end{itemize}
\end{theorem}

The following result is a corollary
Theorem~\ref{theorem:Shokurov-vanishing}
(see~\cite[Theorem~4.1]{Na90}).

\begin{lemma}\label{lemma:rational-tree}
Suppose that $-(K_{X}+B_{X})$ is nef and big and
$\mathrm{dim}(\mathrm{LCS}(X, B_{X}))=1$. Then
\begin{itemize}
\item the locus $\mathrm{LCS}(X, B_X)$ is a connected union of smooth rational curves,%
\item every intersecting irreducible components of the locus $\mathrm{LCS}(X, B_X)$ meet transversally,%
\item the locus $\mathrm{LCS}(X, B_X)$ does not contain a cycle of smooth rational curves.%
\end{itemize}
\end{lemma}

Let $P$ be a point in $X$. Let us consider an effective  divisor
$$
\Delta=\sum_{i=1}^{r}\varepsilon_{i}B_{i}\sim_{\mathbb{Q}} B_{X},%
$$
where $\varepsilon_{i}$ is a non-negative rational number. Suppose
that
\begin{itemize}
\item the divisor $\Delta$ is a $\mathbb{Q}$-Cartier divisor,%
\item the equivalence $\Delta\sim_{\mathbb{Q}} B_{X}$ holds,%
\item the log pair  $(X, \Delta)$ is log canonical in the point $P\in X$.%
\end{itemize}

\begin{remark}
\label{remark:convexity} Suppose that $(X, B_{X})$ is not log
canonical in the point $P\in X$. Put
$$
\alpha=\mathrm{min}\left\{\frac{a_{i}}{\varepsilon_{i}}\ \Big\vert\ \varepsilon_{i}\ne 0\right\}.%
$$
Note that $\alpha$ is well defined, because there is
$\varepsilon_{i}\ne 0$. Then $\alpha<1$, the log pair
$$
\left(X,\ \sum_{i=1}^{r}\frac{a_{i}-\alpha\varepsilon_{i}}{1-\alpha}B_{i}\right)%
$$
is not log canonical in the point $P\in X$, the equivalence
$$
\sum_{i=1}^{r}\frac{a_{i}-\alpha\varepsilon_{i}}{1-\alpha}B_{i}\sim_{\mathbb{Q}} B_{X}\sim_{\mathbb{Q}}\Delta%
$$
holds, and at least one irreducible component of the divisor
$\mathrm{Supp}(\Delta)$ is not contained in
$$
\mathrm{Supp}\left(\sum_{i=1}^{r}\frac{a_{i}-\alpha\varepsilon_{i}}{1-\alpha}B_{i}\right).
$$
\end{remark}

The following result is an easy corollary of
Remark~\ref{remark:convexity}.

\begin{lemma}\label{lemma:fixed-or-mobile}
Let $X$ be a smooth Fano variety such that
$\mathrm{Pic}(X)=\mathbb{Z}[H]$ for some divisor
$H\in\mathrm{Pic}(X)$, and let $G\subset\mathrm{Aut}(X)$ be a
subgroup. Let $\lambda$ be a rational number such that
\begin{itemize}
\item  $\mathrm{lct}(X, D)\geqslant \lambda/n$ for every $G$-invariant divisor $D\in |nH|$;%
\item  $\mathrm{lct}(X, \mathcal{D})\geqslant \lambda/n$ for every $G$-invariant linear subsystem $\mathcal{D}\subset |nH|$ that has no fixed components.%
\end{itemize}
Then
$$
\mathrm{lct}\Big(X,\ G\Big)\geqslant\lambda.
$$
\end{lemma}

\begin{proof}
Suppose that $\mathrm{lct}(X, G)<\lambda$. Then there are a
natural number $n$ and a $G$-invariant linear subsystem
$\mathcal{D}\subset |nH|$ such that the log pair
$$
\left(X, \frac{\lambda}{n}\mathcal{D}\right)
$$
is not log canonical. Put $\mathcal{D}=F+\mathcal{M}$, where $F$
is a fixed part of the linear system $\mathcal{D}$, and
$\mathcal{M}$ is a $G$-invariant linear system that has no fixed
components.

Let $M_{1},\ldots,M_{r}$ be general divisors in $\mathcal{M}$,
where $r\gg 0$. Then
$$
\left(X,\
\frac{\lambda}{n}\left(F+\frac{\sum_{i=1}^{r}M_{i}}{r}\right)\right)
$$
is not log canonical by \cite[Theorem~4.8]{Ko97}.

Since $\mathrm{Pic}(X)=\mathbb{Z}[H]$, we have $F\sim n_1H$ and
$\mathcal{M}\sim n_2H$ for some  $n_1, n_2\in\mathbb{Z}_{>0}$,
such that $n_{1}+n_{2}=n$. By Remark~\ref{remark:convexity}, we
see that the log pair
$$
\left(X,\ \frac{\lambda}{n_{2}r}\sum_{i=1}^{r}M_{i}\right)
$$
is not log canonical, because $F$ is $G$-invariant. Then the log
pair
$$
\left(X,\ \frac{\lambda}{n_{2}}\mathcal{M}\right)
$$
is not log canonical by \cite[Theorem~4.8]{Ko97}, which is a
contradiction.
\end{proof}

The following calculation will be useful in
Section~\ref{section:V5}.

\begin{lemma}\label{lemma:multiplicity-2} Let
$\mathrm{dim}(X)=3$; let $C\subset X$ be an irreducible reduced
curve, and $P\in C$ a point, such that
$$
\mathrm{Sing}\big(C\big)\not\ni P\not\in\mathrm{Sing}\big(X\big).
$$
Let $L\subset X$ be a curve such that $P\not\in\mathrm{Sing}(L)$,
and $D$ a $\mathbb{Q}$-divisor on $X$ such that
$C\subset\mathrm{Supp}(D)\not\supset L$. Assume that $L$ and $C$ are
tangent at $P$. Then
$$
\mathrm{mult}_{P}\Big(D\cdot L\Big)\geqslant 2\mathrm{mult}_C\big(D\big).%
$$
\end{lemma}

\begin{proof}
Let $\pi\colon \tilde{X}\to X$ be a blow-up of the point $P$, let
$E$ be an exceptional divisor of $\pi$. Denote by $\tilde{L}$,
$\tilde{C}$ and $\tilde{D}$ the proper transforms on $\tilde{X}$ of
the curves $L$ and $C$ and the divisor $D$, respectively. Then the
intersection
$$
\tilde{L}\cap\mathrm{Supp}\big(\tilde{D}\big)
$$
contains some point $\tilde{P}\in E$, since $L$ and $C$ are tangent
at $P$. Hence
\begin{multline*}
\mathrm{mult}_{P}\Big(D\cdot
L\Big)=\mathrm{mult}_P(D)+\mathrm{mult}_{\tilde{P}}\Big(\tilde{D}\cdot\tilde{L}\Big)\geqslant
\mathrm{mult}_C(D)+\mathrm{mult}_{\tilde{P}}(\tilde{D})\ge\\
\ge
\mathrm{mult}_C(D)+\mathrm{mult}_{\tilde{C}}(D)=2\mathrm{mult}_C(D).
\end{multline*}
\end{proof}

\section{Veronese double cone}
\label{section:V1}

We'll use the following notation: if $\mathcal{D}$ is a (nonempty)
linear system on the variety $X$, then $\varphi_{\mathcal{D}}$
denotes the rational map defined by $\mathcal{D}$.

Let $V$ be a smooth Fano threefold such that $(-K_V)^3=8$ and
$$
\mathrm{Pic}\big(V\big)=\mathbb{Z}\big[H\big]
$$
for some $H\in\mathrm{Pic}(V)$. Then $V$ is a hypersurface in
$\mathbb{P}(1,1,1,2,3)$ of degree $6$.

The linear system $|H|$ has the only base point $O\in V$ and defines
a rational map
$$
\varphi_{|H|}\colon V\dasharrow\mathbb{P}^2
$$
with irreducible fibers; a general fiber of $\varphi_{|H|}$ is an
elliptic curve.

\begin{remark}
\label{remark:vertical} We will refer to the subvarieties of $V$
that are swept out by the fibers of $\varphi_{|H|}$ as
\emph{vertical} subvarieties.
\end{remark}

Let $G\subset\mathrm{Aut}(V)$ be a subgroup. Note that $G$ is
finite, its action on $V$ extends to $\mathbb{P}(1, 1, 1, 2, 3)$,
and  $G$ naturally acts on $\mathbb{P}(|H|)\cong\mathbb{P}^2$.
Moreover, the following conditions are equivalent:
\begin{itemize}
\item $G$ has no fixed points on $\mathbb{P}(|H|)\cong\mathbb{P}^2$;%
\item $G$ has no invariant lines on $\mathbb{P}(|H|)\cong\mathbb{P}^2$;%
\item  $|H|$ contains no $G$-invariant surfaces;%
\item $|H|$ contains no $G$-invariant pencils (cf. the proof of \cite[Theorem~1.2]{PrMar99});%
\item $V$ is $G$-birationally superrigid (see \cite{Gr03}, \cite{Gr04}).%
\end{itemize}

Let $\mathcal{B}$ be a linear subsystem in $|-K_{X}|$, generated by
divisors of the form
$$
\lambda_{0}x^{2}+\lambda_{1}y^{2}+\lambda_{2}z^{2}+\lambda_{3}xy+\lambda_{4}xz+\lambda_{5}yz=0,
$$
where $x$, $y$ and $z$ are coordinates of weight $1$ in $\P(1, 1, 1, 2, 3)$.
The statement of Theorem~\ref{theorem:V1} is implied by the
following result.

\begin{theorem}
\label{theorem:V1-lct} Supppose that the linear system $\mathcal{B}$
contains no $G$-invariant divisors. Then $\lct(V, G)\ge 1$.
\end{theorem}

Let us prove Theorem~\ref{theorem:V1-lct}. Assume that
$\mathrm{lct}(V, G)<1$. Then the linear system $|H|$ does not
contain $G$-invariant divisors, but there exists an effective
$G$-invariant $\mathbb{Q}$-divisor $D\sim_{\mathbb{Q}} -K_V$, such
that
$$
\mathbb{LCS}\Big(V,\ \lambda D\Big)\neq\varnothing
$$
for some $1>\lambda\in\mathbb{Q}$. The set $\mathrm{LCS}(V, \lambda
D)$ is $G$-invariant.

\begin{lemma}\label{lemma:dim-2}
The set $\mathbb{LCS}(V, \lambda D)$ does not contain divisors.
\end{lemma}

\begin{proof}
Easy.
\end{proof}

\begin{lemma}\label{lemma:dim-1}
The set $\mathbb{LCS}(V, \lambda D)$ does not contain curves.
\end{lemma}

\begin{proof}
Let $C\subset\mathrm{LCS}(V, \lambda D)$ be a $G$-invariant curve.
Then for any point $P\in C$ one has $\mathrm{mult}_P D>1/\lambda$.
Lemma~\ref{lemma:rational-tree} implies that $C$ has a
non-vertical component. Then $\mathrm{deg}(\phi_{|H|}(C))\geqslant
3$, since the linear system $\mathcal{B}$ does not contain
$G$-invariant surfaces.

Let $S$ be a general surface in $|H|$. Put
$$
S\cap C=\Big\{P_{1},\ldots,P_{s}\Big\},
$$
where $P_{1},\ldots,P_{s}$ are distinct points. Then $s\geqslant 3$.
Moreover, one has $s>3$ if $O\in C$. So it is easy to see that one
may assume the following:
\begin{itemize}
\item $O\not\in\{P_{1},P_{2},P_{3}\}$,%
\item $\phi_{|H|}(P_{1}),\phi_{|H|}(P_{2}),\phi_{|H|}(P_{3})$ are distinct points.%
\end{itemize}

The surface $S$ is a del Pezzo surface. One has
$$
D\Big\vert_{S}\sim_{\mathbb{Q}} -2K_{S}
$$
and $-K_{S}^{2}=1$. The log pair $(S, \lambda D\vert_{S})$ is not
log terminal at $P_{1}$, $P_{2}$ and $P_{3}$.
Theorem~\ref{theorem:Shokurov-vanishing} implies that the sequence
$$
\mathbb{C}^{2}=H^{0}\Big(\mathcal{O}_{S}\big(-K_{S}\big)\Big)\longrightarrow H^{0}\Big(\mathcal{O}_{\mathcal{L}(S,\lambda D\vert_{S})}\Big)\longrightarrow H^{1}\Big(\mathcal{I}\big(S,\ \lambda D\big\vert_{S}\big)\otimes\mathcal{O}_{S}\big(-K_{S}\big)\Big)=0%
$$
is exact, since the scheme $\mathcal{L}(S,\lambda D\vert_{S})$ is
zero-dimensional by Lemma~\ref{lemma:dim-2}. In particular, the
support of the subscheme $\mathcal{L}(S,\lambda D\vert_{S})$
contains at most two points, which is a contradiction.
\end{proof}

So $\mathrm{LCS}(V, \lambda D)$ is zero-dimensional.
Theorem~\ref{theorem:connectedness} implies that $\mathrm{LCS}(V,
\lambda D)$ consists of a single point $P\in V$.

\begin{lemma}
\label{lemma:dim-0-not-O} $P=O$.
\end{lemma}

\begin{proof}
Assume that $P\ne O$. Then the $G$-orbit of $P$ is non-trivial,
since so is the $G$-orbit of $\varphi_{|H|}(P)$, which is a
contradiction.
\end{proof}

Let $\pi\colon\bar{V}\to V$ be a blow-up of the point $O$ with an
exceptional divisor $E$; let $\bar{D}$ be a strict transform of
$D$ on $\bar{V}$. Then the log pair
$$
\Big(\bar{V},\ \lambda\bar{D}+\big(\lambda\mathrm{mult}_O(D)-2\big)E\Big).%
$$
is not log canonical in the neighborhood of $E$. On the other hand
one has $\mathrm{mult}_O(D)\leqslant 2$, since otherwise
$\mathrm{Supp}(\bar{D})$ would contain all fibers of the elliptic
fibration $\varphi_{|\pi^*(H)-E|}$. Hence the set
$$
\mathrm{LCS}\Big(\bar{V},\ \lambda\bar{D}+\big(\lambda\mathrm{mult}_O(D)-2\big)E\Big)%
$$
contains some $G$-invariant subvariety $Z\subsetneq E$ and is
contained in $E\cong\mathbb{P}^{2}$.

\begin{lemma}
\label{lemma:V1-final} One has $\mathrm{dim}(Z)=0$.
\end{lemma}

\begin{proof}  Suppose that $\mathrm{dim}(Z)=1$. Let $F$ be a general line
in $E\cong\mathbb{P}^{2}$. Then
$$
2\geqslant
\mathrm{mult}_O\big(D\big)=L\cdot\bar{D}\geqslant\mathrm{deg}(Z)\mathrm{mult}_{Z}\big(\bar{D}\big)>\frac{\mathrm{deg}(Z)}{\lambda}>
\mathrm{deg(Z)}.
$$
Hence $Z$ containes a $G$-invariant line. But $|H|$ does not contain
$G$-invariant surfaces, which gives a contradiction.
\end{proof}

So we see that the $G$-invariant set $\mathrm{LCS}(\bar{V},
\lambda\bar{D})$ consists of a finite number of points. By
Theorem~\ref{theorem:log adjunction-connectedness-theorem} the set
$\mathrm{LCS}(Y, \lambda\bar{D})$ consists of a single point,
since the divisor $-(K_Y+\lambda\bar{D})$ is $\pi$-ample. But $G$
acts on $E$ without fixed points, since $|H|$ contains no
$G$-invariant pencils. The contradiction concludes the proof of
Theorem~\ref{theorem:V1-lct}.

\section{Quintic del Pezzo threefold}
\label{section:V5}

Let $V_{5}$ be a smooth Fano variety such that
$$
\mathrm{Pic}\big(V_{5}\big)=\mathbb{Z}\big[H\big]
$$
and $H^3=5$. One has $-K_{V_5}\sim 2H$ (see, for example,
\cite{IsPr99}). Let $\mathcal{W}\cong\mathbb{C}^3$ be a vector
space endowed with a non-degenerate quadratic form $q$. Then the
variety $V_{5}$ is isomorphic to the variety of triples of
pairwise orthogonal (with respect to $q$) lines in $\mathcal{W}$
(see \cite{IsPr99}). In particular, there is a natural action of
the group $\mathrm{SO}_3(\mathbb{C})$ (or
$\mathrm{SL}_2(\mathbb{C})$) on the variety $V_5$.

\begin{remark}
\label{remark:V5-Aut} One can show that
$\mathrm{Aut}(V_5)=\mathrm{PSL}_2(\mathbb{C})$. By
Remark~\ref{remark:compact-vs-complex} to prove
Theorem~\ref{theorem:V5-lct} it suffices to check that $\lct(V_5,
\PSL_2(\C))=5/6$.
\end{remark}

The variety $V_5$ has a natural
$\mathrm{PSL}_2(\mathbb{C})$-equivariant stratification:
$$
V_5=U\cup\Delta\cup C,
$$
where $U$ is an open orbit which consists of triples of pairwise
distinct lines, $\Delta$ is a two-dimensional orbit which consists
of the triples $(l_1, l_1, l_2)$, where $q(l_1, l_1)=0$ and $q(l_1,
l_2)=0$, and $C$ is a one-dimensional orbit which consists of the
triples $(l, l, l)$, where $q(l, l)=0$.

The linear system $|H|$ defines an embedding
$V_{5}\subset\mathbb{P}^6$. Under this embedding the curve $C$ is
a rational normal curve of degree $6$, and $\Delta$ is swept out
by the lines that are tangent to~$C$.

\begin{lemma}\label{lemma:fixed-divisor}
One has $\mathrm{lct}(V_5, \Delta)=5/6$.
\end{lemma}
\begin{proof}
The surface $\Delta$ is smooth outside $C$ and has a singularity
along $C$ that is locally isomorphic to $T\times\mathbb{A}^1$, where
$T$ is a germ of a cuspidal curve.
\end{proof}

In particular, $\mathrm{lct}(V_5, \PSL_2(\mathbb{C}))\leqslant
5/6$.

\begin{lemma}\label{lemma:mobile-system}
Let $\mathcal{D}\subset|nH|$ be a
$\mathrm{PSL}_2(\mathbb{C})$-invariant linear system on $V_5$,
such that $\Delta\not\subset\nolinebreak\mathrm{Bs}(\mathcal{D})$.
Then $\mathrm{lct}(X, \mathcal{D})\geqslant 1/n$.
\end{lemma}

\begin{proof}
Suppose that $\mathrm{lct}(X, \mathcal{D})<1/n$. Then there exists
a $\mathrm{PSL}_2(\mathbb{C})$-invariant subvariety $Z\subsetneq
X$, such that
$$
\mathrm{mult}_{Z}\big(D\big)>n,
$$
where $D$ is a general divisor in $\mathcal{D}$. Since
$\Delta\not\subset\mathrm{Bs}(\mathcal{D})$, the subvariety $Z$ is
the curve $C$. Let $P$ be a general point of $C$, and $L$ be the
tangent line to $C$ at $P$. Then $L\not\subset\mathrm{Supp}(D)$.
By Lemma~\ref{lemma:multiplicity-2} one has
$$
2n=D\cdot L\geqslant
\mathrm{mult}_{P}\Big(D\cdot L\Big)>2n,%
$$
which is a contradiction.
\end{proof}

Lemmas~\ref{lemma:fixed-or-mobile}, \ref{lemma:fixed-divisor}
and~\ref{lemma:mobile-system} imply that $\mathrm{lct}(V_5,
\PSL_2(\C))\geqslant  5/6$, and hence $\mathrm{lct}(V_5,
\PSL_2(\C))=\nolinebreak 5/6$.

\section{The Mukai--Umemura threefold}
\label{section:MU}

Let $X$ be a smooth Fano threefold such that
$$
\mathrm{Pic}\big(X\big)=\mathbb{Z}\big[-K_{X}\big],
$$
the equality $-K_{X}^{3}=22$ holds and
$\mathrm{Aut}(X)\cong\mathrm{PSL}(2,\mathbb{C})$. It is well known
that the variety having such properties is unique (see
\cite{MuuM83},~\cite{Pr90umn}).

\begin{proposition}
\label{lemma:V22} The equality $\mathrm{lct}(X)=1/2$ holds.
\end{proposition}

\begin{proof}
Let $U\subset\mathbb{C}[x,y]$ be a subspace of forms of degree
$12$. Consider $U\cong\mathbb{C}^{13}$ as the affine~part~of
$$
\mathbb{P}\Big(U\oplus \mathbb{C}\Big)\cong\mathbb{P}^{13},
$$
and let us identify $\mathbb{P}(U)$ with the hyperplane at
infinity.

The natural action of $\mathrm{SL}(2,\mathbb{C})$ on
$\mathbb{C}[x,y]$ induces an action on $\mathbb{P}(U\oplus
\mathbb{C})$. Put
$$
\phi=xy\Big(x^{10}-11x^5y^5-y^{10}\Big)\in U
$$
and consider the closure
$\overline{\mathrm{SL}(2,\mathbb{C})\cdot[\phi+1]}\subset\mathbb{P}(U\oplus
\mathbb{C})$. It follows from \cite{MuuM83} that
$$
X\cong\overline{\mathrm{SL}\big(2,\mathbb{C}\big)\cdot\big[\phi+1\big]},%
$$
and the embedding $X\subset\mathbb{P}(U\oplus
\mathbb{C})\cong\mathbb{P}^{13}$ is induced by $|-K_{X}|$.

It is well known (see \cite[Theorem~5.2.13]{IsPr99}) that the
action of $\mathrm{SL}(2,\mathbb{C})$ on $X$ has the following
orbits:
\begin{itemize}
\item the three-dimensional orbit $\Sigma_3=\mathrm{SL}(2,\mathbb{C})\cdot [\phi+1]$;%
\item the two-dimensional orbit $\Sigma_2=\mathrm{SL}(2,\mathbb{C})\cdot [xy^{11}]$;%
\item the one-dimensional orbit $\Sigma_1=\mathrm{SL}(2,\mathbb{C})\cdot [y^{12}]$.%
\end{itemize}

The orbit $\Sigma_{3}$ is open. The orbit
$\Sigma_{1}\cong\mathbb{P}^{1}$ is closed. One has
$\overline{\Sigma}_2=\Sigma_1\cup\Sigma_2$, and
$$
X\cap \mathbb{P}(U)=\Sigma_1\cup \Sigma_2.
$$

Put $R=X\cap \mathbb{P}(U)$. It follows from \cite{MuuM83} that
\begin{itemize}
\item the surface $R$ is swept out by lines on $X\subset\mathbb{P}^{13}$,%
\item the surface $R$ contains all lines on $X\subset\mathbb{P}^{13}$,%
\item for any lines $L_{1}\subset R\supset L_{2}$ such that $L_{1}\ne L_{2}$, one has $L_{1}\cap L_{2}=\varnothing$,%
\item the surface $R$ is singular along the orbit $\Sigma_1\cong\mathbb{P}^1$,%
\item the normalization of the surface $R$ is isomorphic to $\mathbb{P}^1\times \mathbb{P}^1$,%
\item for every point $P\in \Sigma_{1}$, the surface $R$ is
locally isomorphic to
$$
x^{2}=y^{3}\subset\mathbb{C}^{3}\cong\mathrm{Spec}\Big(\mathbb{C}\big[x,y,z\big]\Big),
$$
which implies that $\mathrm{lct}(X,R)=5/6$.
\end{itemize}

The structure of the surface $R$ can be seen as follows. We see
that
$$
\Sigma_2=\Big\{\Big[\big(ax+by\big)\big(cx+dy\big)^{11}\Big]\ \Big\vert\ ad-bc=1\Big\}\subset\mathbb{P}\big(U\big),%
$$
which implies that there is a birational  morphism
$\nu\colon\mathbb{P}^1 \times \mathbb{P}^1 \to R$ that is defined
by
$$
\nu\colon\big[a:b\big]\times\big[c:d\big]\mapsto \Big[\big(ax+by\big)\big(cx+dy\big)^{11}\Big]\in R,%
$$
which is a normalization of the surface $R$.

Let $V_{5}$ be a smooth Fano threefold such that
$$
-K_{V_{5}}\sim 2H
$$
and $H^{3}=5$, where $H$ is a Cartier divisor on $V_{5}$. Then
$|H|$ induces an embedding $X\subset\mathbb{P}^{6}$ (see
Section~\ref{section:V5}).

Let $L\cong\mathbb{P}^{1}$ be a line on $X$. Then
$$
\mathcal{N}_{L\slash
X}\cong\mathcal{O}_{\mathbb{P}^{1}}(-2)\oplus\mathcal{O}_{\mathbb{P}^{1}}(1).
$$

Let $\alpha_{L}\colon U_{L}\to X$ be a blow up of the line $L$,
and let $E_{L}$ be the exceptional divisor of $\alpha_{L}$. Then
it follows from Theorem~4.3.3 in \cite{IsPr99} that there is a
commutative diagram
$$
\xymatrix{
U_{L}\ar@{->}[d]_{\alpha_{L}}\ar@{-->}[rr]^{\rho_{L}}&&W_{L}\ar@{->}[d]^{\beta_{L}}\\%
X\ar@{-->}[rr]_{\psi_{L}}&&V_{5}}
$$
where $\rho_{L}$ is a flop in the exceptional section of
$E\cong\mathbb{F}_{3}$, the morphism $\beta_{L}$ contracts a
surface $D_{L}\subset W_{L}$ to a smooth rational curve of degree
$5$, and $\psi_{L}$ is a double projection from the line $L$.

Let $\bar{D}_{L}\subset X$ be the proper transform of the surface
$D_{L}$. Then
$$
\mathrm{mult}_{L}\big(\bar{D}_{L}\big)=3
$$
and $\bar{D}_{L}\sim -K_{X}$. It follows from \cite{Fur90} that
$X\setminus\bar{D}_{L}\cong\mathbb{C}^{3}$.

It follows from \cite{Fur92} that there is an open subset
$\breve{D}_{L}\subset\bar{D}_{L}$ that is given by
$$
\mu_{0}x^{4}+\Big(\mu_{1}yz+\mu_{2}z^{3}\Big)x^{3}+\Big(\mu_{3}y^{3}+\mu_{4}y^{2}z^{2}+\mu_{5}yz^{4}\Big)x^{2}+\Big(\mu_{6}y^{4}z+\mu_{7}y^{3}z^{3}\Big)x+\mu_{8}y^{6}+\mu_{9}y^{5}z^{2}=0%
$$
in $\mathbb{C}^{3}\cong\mathrm{Spec}(\mathbb{C}[x,y,z])$, where
the point $L\cap \Sigma_{1}\in\breve{D}_{L}$ is given by the
equations $x=y=z=0$, and
$$
\left.%
\aligned
&\mu_{0}=-2^{8}5^{2},\ \mu_{1}=2^{9}3^{3}5,\ \mu_{2}=-2^{6}3^{4}5,\ \mu_{3}=-2^{8}3^{3}7,\ \mu_{4}=-2^{4}3^{4}127,\,\\%
&\mu_{5}=2^{9}3^{5},\ \mu_{6}=2^{2}3^{6}89,\ \mu_{7}=-2^{8}3^{6},\ \mu_{8}=-3^{6}5^{3},\ \mu_{9}=2^{5}3^{7}.\\%
\endaligned\right.%
$$

Put $O_{L}=\Sigma_{1}\cap L$. Then
$\mathrm{mult}_{O_{L}}(\bar{D}_{L})=4$, and it follows from
\cite[Proposition~8.14]{Ko97}  that
$$
\mathrm{LCS}\left(X,\ \frac{1}{2}\bar{D}_{L}\right)=O_{L}
$$
and $\mathrm{lct}(X, \bar{D}_{L})=1/2$. Thus, we see that
$\mathrm{lct}(X)\leqslant 1/2$.

Suppose that $\lct(X)<1/2$. Then there exists an effective
$\mathbb{Q}$-divisor
$$
D\sim_{\mathbb{Q}} -K_{X},
$$
such that the log pair $(X,\lambda D)$ is not log canonical for
some positive rational number $\lambda<1/2$. By
Remark~\ref{remark:convexity}, we may assume that
$R\not\subset\mathrm{Supp}(D)$, because $\mathrm{lct}(X,R)=5/6$.

Let $C$ be a line in $X$ such that $C\not\subset\mathrm{Supp}(D)$.
Then
$$
1=D\cdot C\geqslant\mathrm{mult}_{O_{C}}\big(D\big)\mathrm{mult}_{O_{C}}\big(C\big)=\mathrm{mult}_{O_{C}}\big(D\big),%
$$
which implies that $O_{C}\not\in\mathrm{LCS}(X,\lambda D)$. In
particular, we see that $\Sigma_{1}\not\in\mathrm{LCS}(X,\lambda
D)$.

Let $\Gamma$ be an irreducible curve in $\mathrm{Supp}(D)$ such
that $O_{C}\in\Gamma$. Then
$$
\mathrm{mult}_{\Gamma}\left(\frac{1}{2}\bar{D}_{C}+\lambda D\right)=\frac{\mathrm{mult}_{\Gamma}\big(\bar{D}_{C}\big)}{2}+\lambda\mathrm{mult}_{\Gamma}\big(D\big)\leqslant\frac{\mathrm{mult}_{\Gamma}\big(\bar{D}_{C}\big)}{2}+\lambda\mathrm{mult}_{O_{C}}\big(D\big)<1,%
$$
because $\lambda<1/2$ and $\mathrm{Sing}(\bar{D}_{C})=C$, because
$\bar{D}_{C}\ne R$. Thus, we see that
$$
\Gamma\not\subseteq\mathrm{LCS}\left(X,\ \frac{1}{2}\bar{D}_{C}+\lambda D\right)\supseteq\mathrm{LCS}\Big(X, \lambda D\Big)\cup O_{C},%
$$
which is impossible by Theorem~\ref{theorem:connectedness},
because $O_{C}\not\in\mathrm{LCS}(X,\lambda D)$ and $\lambda<1/2$.
\end{proof}

\begin{remark}\label{remark:Donaldson}
It follows from \cite{Do07} that
$$
\mathrm{lct}\Big(X,\
\mathrm{SO}_3\big(\mathbb{R}\big)\Big)=\frac{5}{6},
$$
which implies, in particular, that $X$ has a K\"ahler--Einstein
metric. This inequality can be obtained by arguing as in the proof
of Theorem~\ref{theorem:V5-lct} (the only difference is that we do
not need to use Lemma~\ref{lemma:multiplicity-2} here).
\end{remark}


\begin{thebibliography}{137}




\bibitem{ArGhPi06}
C.\,Arezzo, A.\,Ghigi, G.\,Pirola, \emph{Symmetries, quotients and K\"ahler--Einstein metrics}\\
Journal fur die Reine und Angewandte Mathematik \textbf{591} (2006), 177--200%

\smallskip

\bibitem{Ch01b}
I.\,Cheltsov, \emph{Log canonical thresholds on hypersurfaces}\\
Sbornik: Mathematics \textbf{192} (2001), 1241--1257%

\smallskip

\bibitem{Ch07a}
I.\,Cheltsov, \emph{Fano varieties with many selfmaps}\\
Advances in Mathematics \textbf{217} (2008), 97--124%

\smallskip

\bibitem{Ch07b}
I.\,Cheltsov, \emph{Log canonical thresholds of del Pezzo surfaces}\\
Geometric and Functional Analysis \textbf{18} (2008), 1118--1144
\smallskip

\bibitem{Ch08d}
I.\,Cheltsov, \emph{Extremal metrics on two Fano varieties}\\
Sbornik: Mathematics \textbf{200} (2009), 97--136

\smallskip

\bibitem{ChPaWo}
I.\,Cheltsov, J.\,Park, J.\,Won, \emph{Log canonical thresholds of certain Fano hypersurfaces}\\
arXiv:math.AG/0706.0751 (2007)%


\smallskip

\bibitem{ChSh08}
I.\,Cheltsov, C.\,Shramov, \emph{Log canonical thresholds of smooth Fano threefolds.\\ With an appendix by Jean-Pierre Demailly}\\
Russian Mathematical Surveys, to appear

\smallskip

\bibitem{DeKo01}
J.-P.\,Demailly, J.\,Koll\'ar, \emph{Semi-continuity of complex singularity exponents\\ and K\"ahler--Einstein metrics on Fano orbifolds}\\
Annales Scientifiques de l'ìole Normale Sup\'erieure \textbf{34} (2001), 525--556%

\smallskip

\bibitem{Do07}
S.\,Donaldson, \emph{A note on the $\alpha$-invariant of the Mukai--Umemura $3$-fold}\\
arXiv:math.AG/0711.4357 (2007)%

\smallskip

\bibitem{Fur90}
M.\,Furushima, \emph{Complex analytic compactification of $\mathbb{C}^3$}\\
Compositio Mathematica \textbf{76} (1990), 163--196%

\smallskip

\bibitem{Fur92}
M.\,Furushima, \emph{Mukai--Umemura's example of the Fano threefold with genus $12$\\ as a compactification of $\mathbb{C}^3$}\\
Nagoya Mathematical Journal \textbf{127} (1992), 145--165%

\smallskip


\bibitem{Gr03}
M.\,Grinenko, \emph{On the double cone over the Veronese surface}\\
Izvestiya: Mathematics \textbf{67} (2003), 421--438

\smallskip

\bibitem{Gr04}
M.\,Grinenko, \emph{Mori structures on a Fano threefold of index $2$ and degree $1$}\\%
Proceedings of the Steklov Institute of Mathematics \textbf{246} (2004), 103--128%

\smallskip

\bibitem{Hw06b}
J.-M.\,Hwang, \emph{Log canonical thresholds of divisors on Fano manifolds of Picard rank $1$}\\
Compositio Mathematica \textbf{143} (2007), 89--94%

\smallskip


\bibitem{IF00}
A.\,R.\,Iano-Fletcher, \emph{Working with weighted complete intersections}\\
L.M.S. Lecture Note Series \textbf{281} (2000), 101--173%

\smallskip

\bibitem{IsPr99}
V.\,Iskovskikh, Yu.\,Prokhorov, \emph{Fano varieties}\\
Encyclopaedia of Mathematical Sciences \textbf{47} (1999) Springer, Berlin%

\smallskip

\bibitem{Ko97}
J.\,Koll\'ar, \emph{Singularities of pairs}\\
Proceedings of Symposia in Pure Mathematics \textbf{62} (1997), 221--287%

\smallskip

\bibitem{La04}
R.\,Lazarsfeld, \emph{Positivity in algebraic geometry} \textbf{II}\\
Springer-Verlag, Berlin, 2004%

\smallskip

\bibitem{MuuM83}
S.\,Mukai, H.\,Umemura, \emph{Minimal rational threefolds}\\
Lecture Notes in Mathematics \textbf{1016} (1983), 490--518%

\smallskip

\bibitem{Na90}
A.\,Nadel, \emph{Multiplier ideal sheaves and K\"ahler--Einstein metrics of positive scalar curvature}\\
Annals of Mathematics \textbf{132} (1990), 549--596%

\smallskip

\bibitem{Pr90umn}
Yu.\,Prokhorov, \emph{Automorphism groups of Fano $3$-folds}\\
Russian Mathematical Surveys \textbf{45} (1990), 222--223

\smallskip

\bibitem{PrMar99}
Yu.\,Prokhorov, D.\,Markushevich, \emph{Exceptional quotient singularities}\\
American Journal of Mathematics \textbf{121} (1999), 1179--1189

\smallskip

\bibitem{Pu04d}
A.\,Pukhlikov, \emph{Birational geometry of Fano direct products}\\
Izvestiya: Mathematics \textbf{69} (2005), 1225--1255%

\smallskip


\bibitem{Sho93}
V.\,Shokurov, \emph{Three-fold log flips}\\
Russian Academy of Sciences, Izvestiya Mathematics \textbf{40} (1993), 95--202 %

\smallskip

\bibitem{Ti87}
G.\,Tian, \emph{On K\"ahler--Einstein metrics on certain K\"ahler manifolds with $c_{1}(M)>0$}\\
Inventiones Mathematicae \textbf{89} (1987), 225--246%

\smallskip

\bibitem{WaZhu04}
X.\,Wang, X.\,Zhu, \emph{K\"ahler--Ricci solitons on toric manifolds with positive first Chern class}\\
Advances in Mathematics \textbf{188} (2004), 87--103%

\smallskip

\bibitem{Zh06}
Q.\,Zhang, \emph{Rational connectedness of log $\mathbb{Q}$-Fano varietiess}\\
Journal fur die Reine und Angewandte Mathematik \textbf{590} (2006), 131--142%



\end{thebibliography}
\end{document}